\documentclass[11pt,leqno]{article}
\usepackage{amsthm,amsfonts,amssymb,amsmath,oldgerm}
\usepackage{epsfig}
\numberwithin{equation}{section}
\usepackage[thinlines]{easybmat}

\newcommand{\DDD}{D3'}

\newcommand{\bu}{\bar{u}}

\newcommand{\txi}{\tilde{\xi}}
\newcommand{\x}{{\xi}_1} 

\newcommand{\td}{\tilde}

\newcommand{\hv}{\hat{v}}
\newcommand{\curl}{\hbox{ \rm curl }}

\def\txi{{\tilde \xi}}

\def\eps{\varepsilon }

\newcommand\R{\mathbb R}

\def\eps{\varepsilon}

\newcommand\br{\begin{remark}}
\newcommand\er{\end{remark}}
\newcommand\bp{\begin{pmatrix}}
\newcommand\ep{\end{pmatrix}}
\newcommand\be{\begin{equation}}
\newcommand\ee{\end{equation}}
\newcommand\ba{\begin{equation}\begin{aligned}}
\newcommand\ea{\end{aligned}\end{equation}}

\newcommand{\bap}{\begin{app}}
\newcommand{\eap}{\end{app}}
\newcommand{\begs}{\begin{exams}}
\newcommand{\eegs}{\end{exams}}
\newcommand{\beg}{\begin{example}}
\newcommand{\eeg}{\end{exaplem}}
\newcommand{\bpr}{\begin{proposition}}
\newcommand{\epr}{\end{proposition}}
\newcommand{\bt}{\begin{theorem}}
\newcommand{\et}{\end{theorem}}
\newcommand{\bc}{\begin{corollary}}
\newcommand{\ec}{\end{corollary}}
\newcommand{\bl}{\begin{lemma}}
\newcommand{\el}{\end{lemma}}
\newcommand{\bd}{\begin{definition}}
\newcommand{\ed}{\end{definition}}
\newcommand{\brs}{\begin{remarks}}
\newcommand{\ers}{\end{remarks}}

\newtheorem{theo}{Theorem}[section]
\newtheorem{prop}[theo]{Proposition}
\newtheorem{cor}[theo]{Corollary}
\newtheorem{lem}[theo]{Lemma}

\newtheorem{exams}[theo]{Examples}

\numberwithin{equation}{section}

\newcommand{\RR}{{\mathbb R}}

\newcommand{\CC}{{\mathbb C}}

\newcommand{\const}{\text{\rm constant}}

\newtheorem{theorem}{Theorem}[section]
\newtheorem{proposition}[theorem]{Proposition}
\newtheorem{corollary}[theorem]{Corollary}
\newtheorem{lemma}[theorem]{Lemma}
\newtheorem{definition}[theorem]{Definition}

\newtheorem{example}[theorem]{Example}
\newtheorem{remark}[theorem]{Remark}

\newcommand{\RM}{\mathbb{R}}

\title{
Nonlinear stability of periodic traveling wave solutions
of systems of viscous conservation laws in the generic case}

\author{\sc \small
Mathew A. Johnson\thanks{Indiana University, Bloomington, IN 47405;
matjohn@indiana.edu: Research of M.J. was partially supported by an NSF Postdoctoral Fellowship under NSF grant DMS-0902192.}
~~~~~
Kevin Zumbrun\thanks{Indiana University, Bloomington, IN 47405;
kzumbrun@indiana.edu:
Research of K.Z. was partially supported
under NSF grants no. DMS-0300487 and DMS-0801745.
 }}
\begin{document}

\maketitle


\begin{center}
\end{center}

\begin{center}
{\bf 2000 MR Subject Classification}: 35B35.
\end{center}


\begin{abstract}
Extending previous results of Oh--Zumbrun and Johnson--Zumbrun,
we show that spectral stability implies linearized and
nonlinear stability of spatially periodic traveling-wave
solutions of viscous systems of conservation laws
for systems of generic type, removing
a restrictive assumption
that wave speed be constant
to first order along the manifold of nearby periodic solutions.
\end{abstract}


\bigbreak

\section{Introduction }\label{intro}

Nonclassical viscous conservation laws
arising in multiphase fluid and solid mechanics
exhibit a rich variety of traveling wave phenomena,
including homoclinic (pulse-type) and periodic solutions
along with the standard heteroclinic (shock, or front-type)
solutions \cite{GZ,Z6,OZ1,OZ2}.
Here, we investigate stability of spatially periodic traveling waves:
specifically, sufficient conditions for stability of the wave.

In previous work \cite{OZ4,JZ3}, we showed that {strong spectral stability
in the sense of Schneider \cite{S1,S2,S3}
implies linearized and nonlinear
$L^1 \cap H^K\to L^\infty$ stability} in all dimensions $d\ge 1$.
However, as pointed out in \cite{OZ1,Se1}, the conditions
of Schneider are {\it nongeneric in the conservation law setting},
implying the restrictive condition
that wave speed be {constant to first order} along the
manifold of nearby periodic solutions.
Indeed, it was shown in \cite{OZ2} that failure of this condition
implies a degradation in the decay rates of the Green
function of the linearized equations about the periodic wave,
suggesting that nonlinear stability would be unlikely in the
general (nonstationary wave speed) case in dimension $d=1$.

In this paper, we show that these difficulties are only apparent,
and that, somewhat surprisingly,
spectral stability implies nonlinear stability even if
this additional condition on wave speeds is dropped.
More precisely, we show that small $L^1 \cap H^s$
perturbations of a planar periodic solution $u(x,t)\equiv \bar u(x_1)$
(without loss of generality taken stationary)
converge at Gaussian rate in $L^p$, $p\ge 2$ to a modulation
\be\label{mod}
\bar u(x_1-\psi(x,t))
\ee
of the unperturbed wave,
where $x=(x_1,\tilde x)$, $\tilde x=(x_2, \dots, x_d)$, and
$\psi$ is a scalar function whose $x$- and $t$-gradients
decay at Gaussian rate in all $L^p$, $p\ge 2$,
but which itself decays more slowly by a factor $t^{1/2}$;
in particular, $\psi$ is merely bounded in $L^\infty$ for dimension $d=1$.

In proving this result, we make crucial use of the tools developed
in \cite{OZ4,JZ3}, in particular, a key nonlinear cancellation argument
of \cite{JZ3}.
The key new observation making possible the treatmen
of the generic case is a careful study
of the Bloch perturbation expansion about frequency $\xi=0$,
motivated by relations to the Whitham averaged system
observed in \cite{Se1,OZ3,JZ1,JZB}.

It was shown in \cite {Se1,OZ3}
that the low-frequency dispersion relation near zero of the linearized
operator about a periodic solution $\bar u$ agrees to first order
with that of the linearization about a constant state of the Whitham
averaged system
\ba \label{e:wkb}
\partial_t M + \sum_j \partial_{x_j}F^j &=0,   \\
\partial_t (\Omega N) + \nabla_x  (\Omega S)&=0,
\ea
where $M\in \RR^n$ denotes the average over one period,
$F^j$ the average of an associated flux,
$\Omega=|\nabla_x \Psi|\in \RR^1$ the frequency, $S=-\Psi_t/|\nabla_x \Psi|\in \RR^1$ the speed
$s$, and $N=\nabla_x \Psi/|\nabla_x \Psi|\in \RR^d$ the normal $\nu$ associated
with nearby periodic waves,
with an additional constraint
\begin{equation}\label{con}
\curl (\Omega N)=\curl \nabla_x \Psi \equiv 0.
\end{equation}
As noted in \cite{Se1,OZ3}, this implies both that the eigenvalues
$\lambda_j(\xi)$ bifurcating from $\lambda=0$ at $\xi=0$
are $C^1$ along rays through the origin, and that weak hyperbolicity
(reality of characteristics of \eqref{e:wkb}--\eqref{con}) is
necessary for spectral or linearized stability.

As noted in \cite{JZB}, there is a deeper analogy between
the low-frequency linearized dispersion relation and the Whitham
averaged system at the structural level,
suggesting a useful rescaling of the low-frequency perturbation problem.
It is this intuition that motivates our derivation of sharp low-frequency
estimates crucial to the analysis of nonlinear stability.
With these estimates in place, the rest of the argument goes exactly
as in \cite{JZ3,OZ4}.

\subsection{Equations and assumptions}\label{s:equations}
Consider a parabolic system of conservation laws
\be
u_t + \sum_j f^j(u)_{x_j} = \Delta_x u,
\label{eqn:1conslaw}
\ee
$u \in {\cal U} (\hbox{open}) \in \R^n$,  $f^j \in \R^n$,
$x \in \R^d$, $d\ge 1$, $t \in \R^+$,
and a periodic traveling wave solution
\be
u=\bar{u}(x\cdot \nu -st),
\ee
of period $X$, satisfying the traveling-wave
ODE
$
\bar u''=
( \sum_j \nu_j f^j(\bar u) )'-s\bar u'
$
with boundary conditions $ \bar u(0) = \bar u(X)=:u_0.	$
Integrating, we obtain a first-order profile equation
\be
\bar u'= \sum_j \nu_j f^j(\bar u) -s \bar u -q,
\label{e:profile}
\ee
where $(u_0,q,s,\nu,X)\equiv \const$.
Without loss of generality take $\nu=e_1$, $s=0$,
so that $\bar u=\bar{u}(x_1)$ represents a stationary solution
depending only on $x_1$.

Following \cite{Se1,OZ3,OZ4}, we assume:

(H1) $f^j\in C^{K+1}$,
$K\ge [d/2]+4$.

(H2) The map $H: \,
\R \times {\cal U} \times \R \times S^{d-1} \times \R^n  \rightarrow \R^n$	
taking
$(X; a, s, \nu, q)  \mapsto u(X; a, s, \nu, q)-a$
is full rank at $(\bar{X}; \bar{u}(0), 0, e_1, \bar{q})$,
where $u(\cdot;\cdot)$ is the solution operator of \eqref{e:profile}.

Conditions (H1)--(H2) imply that the set of periodic solutions
in the vicinity of $\bar u$ form a
smooth $(n+d+1)$-dimensional manifold
$\{\bar u^a(x\cdot \nu(a)-\alpha-s(a)t)\}$,
with $\alpha\in \RR$, $a\in \RR^{n+d}$.

\subsubsection{Linearized equations}\label{evans}

Linearizing (\ref{eqn:1conslaw}) about
$\bar{u}(\cdot)$, we obtain
\be
v_t = Lv := \Delta_x v -\sum(A^j v)_{x_j}, \label{e:lin}
\ee
where coefficients
$A^j:= Df^j(\bu)$
are now periodic functions of $x_1$.
Taking the Fourier transform in the transverse coordinate $\td{x}=
(x_2, \cdots, x_d)$, we obtain
\ba
\hv_t = L_{\td{\xi}}\hv
& = \hv_{x_1,x_1}
-(A^1 \hv)_{x_1}
 - i \sum_{j\ne 1}A^j \xi_j \hv
	- \sum_{j\ne 1} \xi_j^2  \hv,
\label{e:fourier}
\ea
where $\td{\xi}=(\xi_2, \cdots, \xi_d)$ is the transverse frequency
vector.

\subsubsection{Bloch--Fourier decomposition
and stability conditions}\label{bloch}

Following \cite{G,S1,S2,S3}, we define the family of operators
\be
L_{\xi} = e^{-i \xi_1 x_1} L_{\txi}  e^{i \xi_1 x_1}
\label{e:part}
\ee
operating on the class of $L^2$ periodic functions on $[0,X]$;
the $(L^2)$ spectrum
of $L_{\txi}$ is equal to the union of the
spectra of all $L_{\xi}$ with $\xi_1$ real with associated
eigenfunctions
\be
w(x_1, \txi,\lambda) := e^{i \xi_1 x_1} q(x_1, \x, \txi, \lambda),
\label{e:efunction}
\ee
where $q$, periodic, is an eigenfunction of $L_{\xi}$.
By continuity of spectrum,
and discreteness of the spectrum of the elliptic operators $L_\xi$ on
the compact domain $[0,X]$,
we have that the spectra of $L_{\xi}$
may be described as the union of countably many continuous
surfaces $\lambda_j(\xi)$.

Without loss of generality taking $X=1$,
recall now the {\it Bloch--Fourier representation}
\be\label{Bloch}
u(x)=
\Big(\frac{1}{2\pi }\Big)^d \int_{-\pi}^{\pi}\int_{\R^{d-1}}
e^{i\xi\cdot x}\hat u(\xi, x_1)
d\xi_1\, d\tilde \xi
\ee
of an $L^2$ function $u$, where
$\hat u(\xi, x_1):=\sum_k e^{2\pi ikx_1}\hat u(\xi_1+ 2\pi k,\tilde \xi)$
are periodic functions of period $X=1$, $\hat u(\tilde \xi)$
denoting with slight abuse of notation the Fourier transform of $u$
in the full variable $x$.
By Parseval's identity, the Bloch--Fourier transform
$u(x)\to \hat u(\xi, x_1)$ is an isometry in $L^2$:
\be\label{iso}
\|u\|_{L^2(x)}=
\|\hat u\|_{L^2(\xi; L^2(x_1))},
\ee
where $L^2(x_1)$ is taken on $[0,1]$ and $L^2(\xi)$
on $[-\pi,\pi]\times \R^{d-1}$.
Moreover, it diagonalizes the periodic-coefficient operator $L$,
yielding the {\it inverse Bloch--Fourier transform representation}
\be\label{IBFT}
e^{Lt}u_0=
\Big(\frac{1}{2\pi }\Big)^d \int_{-\pi}^{\pi}\int_{\R^{d-1}}
e^{i\xi \cdot x}e^{L_\xi t}\hat u_0(\xi, x_1)
d\xi_1\, d\tilde \xi
\ee
relating behavior of the linearized system to
that of the diagonal operators $L_\xi$.

Loosely following \cite{OZ4}, we assume along with (H1)--(H2) the
{\it strong spectral stability} conditions:

(D1) $\sigma(L_\xi) \subset \{ \hbox{\rm Re} \lambda <0 \} $ for $\xi\ne 0$.

(D2) $\hbox{\rm Re} \sigma(L_{\xi}) \le -\theta |\xi|^2$, $\theta>0$,
for $\xi\in \R^d$ and $|\xi|$ sufficiently small.

(\DDD) $\lambda=0$ is
an eigenvalue
of $L_{0}$ of multiplicity exactly $n+1$.\footnote{
The zero eigenspace of $L_0$ is at least $(n+1)$-dimensional by linearized existence theory and (H2),
and hence $n+1$ is the minimal multiplicity; see \cite{Se1}.}

As shown in \cite{OZ3}, (H1)-(H2) and (D1)--(\DDD)
imply that there exist $n+1$ smooth eigenvalues
\be\label{e:surfaces}
\lambda_j(\xi)= -i a_j(\xi)+o(|\xi|)
\ee
of $L_\xi$ bifurcating from $\lambda=0$ at $\xi=0$, where
$-i a_j$ are homogeneous degree one functions;
see Lemma \ref{blochfacts} below.

As in \cite{OZ4}, we make the further nondegeneracy hypothesis:

(H3) The functions $ a_j(\xi)$ in \eqref{e:surfaces} are distinct.

\noindent
The functions $a_j$ may be seen to be the characteristics
associated with the Whitham averaged system \eqref{e:wkb}--\eqref{con}
linearized about the values of $M$, $S$, $N$, $\Omega$ associated
with the background wave $\bar u$; see \cite{OZ3,OZ4}.
Thus, (D1) implies weak hyperbolicity
of (\ref{e:wkb})--\eqref{con} (reality of $a_j$),
while (H1) corresponds to strict hyperbolicity.

\br\label{Drmks}
\textup{
Condition (\DDD) is a
weakened
version of the condition (D3)
of \cite{OZ4,JZ3} that $\lambda=0$ be a {\it semisimple} eigenvalue
of $L_{0}$ of minimal multiplicity $n+1$, which implies \cite{OZ1,OZ2,Se1}
the special property that wave speed be stationary at $\bar u$ along
the manifold of nearby periodic solutions.
The stronger conditions (D1)--(D3) are exactly the spectral assumptions of
\cite{S1,S2,S3} introduced by Schneider in the reaction-diffusion case.
Conditions (D1)--(D3) (resp. (D1)--(\DDD))
 correspond to ``dissipativity'' of the
large-time behavior of the linearized system \cite{S1,S2,S3}.
}
\er

\subsection{Main result}

With these preliminaries, we can now state our main result.

\begin{theo}\label{main}
Assuming (H1)--(H3) and (D1)--(\DDD),
for some $C>0$ and $\psi \in W^{K,\infty}(x,t)$,
\ba\label{eq:smallsest}
|\tilde u-\bar u(\cdot -\psi)|_{L^p}(t)&\le
C(1+t)^{-\frac{d}{2}(1-1/p)}
|\tilde u-\bar u|_{L^1\cap H^K}|_{t=0},\\
|\tilde u-\bar u(\cdot -\psi)|_{H^K}(t)&\le
C(1+t)^{-\frac{d}{4}}
|\tilde u-\bar u|_{L^1\cap H^K}|_{t=0},\\
|(\psi_t,\psi_x)|_{W^{K+1,p}}&\le
C(1+t)^{-\frac{d}{2}(1-1/p)}
|\tilde u-\bar u|_{L^1\cap H^K}|_{t=0}\\
\ea
for all $t\ge 0$, $p\ge 2$, $d\ge 1$, and
\ba\label{eq:stab}
|\tilde u-\bar u|_{ L^p}(t), \; |\psi(t)|_{L^p}&\le
C(1+t)^{-\frac{d}{2}(1-\frac{1}{p}) + \frac{1}{2}}
|\tilde u-\bar u|_{L^1\cap H^K}|_{t=0}
\ea
for all $t\ge 0$ and
$p=\infty$ or $p\ge 2$ and $d\ge 3$,
for solutions $\tilde u$ of \eqref{eqn:1conslaw} with
$|\tilde u-\bar u|_{L^1\cap H^K}|_{t=0}$ sufficiently small.
In particular, $\bar u$ is nonlinearly bounded
$L^1\cap H^K\to L^\infty$ stable
for $d\ge 1$, asymptotically $L^1\cap H^K\to L^\infty$ stable for $d\ge 2$,
and asymptotically $L^1\cap H^K\to H^K$ stable for $d\ge 3$.
\end{theo}


\br
In Theorem \ref{main}, derivatives in $x\in\RM^d$ for $d\geq 2$ refer to total derivatives.  Moreover, unless specified by an
appropriate index, throughout this paper derivatives in spatial variable $x$ will always refer to the total derivative
of the function.
\er

In dimension one, Theorem \ref{main} asserts only
bounded $L^1\cap H^K \to L^\infty$ stability,
a very weak notion of stability.
The bounds \eqref{eq:smallsest}--\eqref{eq:stab}
agree for dimension $d=1$ with those obtained in \cite{JZ3}
in the stationary wave speed case that (D3) holds in place
of (\DDD), but for higher dimensions are weaker by roughly
factor $t^{1/2}$.

\br\label{nonlinp}
\textup{
In dimension $d=1$, it is straightforward to show that the results of
Theorem \ref{main} extend to all $1\le p\le \infty$ using the pointwise
techniques of \cite{OZ2}; see Remark \ref{lowp}.
}
\er

\subsection{Discussion and open problems}\label{s:discussion}

The proof of Theorem \ref{main} largely completes the line of
investigation carried out in \cite{OZ2,Se1,OZ3,OZ4,JZ3},
showing that spectral stability implies linear and nonlinear stability
of planar spatially periodic traveling waves.
The corresponding spectral stability problem has been studied
analyticially in \cite{OZ1,Se1,OZ3}, yielding various necessary conditions,
and by numerical Evans function investigation in \cite{OZ1}.
An interesting direction for further study
would be more systematic numerical investigation along the lines
of \cite{BLZ,HLyZ1,HLyZ2,BHZ,BLZ} in the viscous shock wave case.
A second interesting open problem would be to extend the results
for planar waves to the case of solutions with multiple periods,
as considered in the reaction--diffusion setting in \cite{S1,S2,S3}.

The key to the nonlinear analysis in critical dimensions $d=1,\, 2$,
as in \cite{JZ3,S1,S2,S3}, is to subtract out
a slower-decaying part of the solution described by an appropriate
modulation equation and show that the residual decays sufficiently
rapidly to close a nonlinear iteration.
Note that the modulated approximation
$\bar u(x_1-\psi(x,t))$ of \eqref{mod} is not the full Ansatz
$\bar u^a(\Psi(x,t))$,
$\Psi(x,t):=x_1-\psi(x,t)$,
associated with the Whitham averaged system \eqref{e:wkb}--\eqref{con},
where $\bar u^a$ is the manifold of periodic
solutions near $\bar u$ introduced below (H2), but only the translational
part not involving perturbations $a$ in the profile.
(See \cite{OZ3} for the derivation of Ansatz and \eqref{e:wkb}--\eqref{con}.)
That is, we don't need to separate out all variations along the
manifold of periodic solutions,
but only the special variations connected with translation invariance.

This can be understood heuristically by the observation that \eqref{e:wkb}
indicates that variables $a$, $\nabla_x \Psi$ are roughly comparable,
which would suggest, by the diffusive behavior $\Psi>>\nabla_x \Psi$,
that $a$ is neglible with respect to $\Psi$.
Indeed, this heuristic argument translates rigorously to our ultimate
computation of linearized behavior leading to the final result;
see Section \ref{prep} and Remark \ref{r:whitrelation}.
In this respect, the connection to the Whitham system is somewhat clearer in
the generic case considered here than in the quasi-Hamiltonian
case treated previously in \cite{OZ2,OZ4,JZ3}.\footnote{
In the degenerate case that the stronger condition (D3) holds,
i.e., wave speed is stationary at $\bar u$, the situation is
somewhat more complicated, and these relations break down;
see \cite{JZ3} for further discussion.}

It would be interesting to better understand
the connection between the Whitham averaged system (or suitable higher-order
correction) and behavior at the nonlinear level, as explored at the
linear level in \cite{OZ3,OZ4,JZ1,JZB}.
As discussed further in \cite{OZ3}, another interesting problem
would be to try to rigorously justify the WKB expansion for the
related vanishing viscosity problem, in the spirit of \cite{GMWZ1,GMWZ2}.

\section{Spectral preparation}\label{prep}
We begin by a careful study of the Bloch perturbation expansion at $\xi=0$.

\begin{lemma}\label{blochfacts}
Assuming (H1)--(H3), (D1)--(\DDD), the eigenvalues
$\lambda_j(\xi/|\xi|, \xi)$ of $L_\xi$
are analytic functions of $\xi/|\xi|$ and $|\xi|$.
Suppose further that $0$ is a nonsemisimple eigenvalue of $L_0$, i.e.,
(D3') holds, but not (D3).
Then, the Jordan structure of the zero eigenspace of $L_0$ consists
of an $n$-dimensional kernel and a single Jordan chain of height $2$,
where the left kernel of $L_0$ is the $n$-dimensional subspace
of constant functions, and $\bar u'$ spans the right
eigendirection lying at the base of the Jordan chain.
Moreover, for $|\xi|$ sufficiently small,
 there exist right and left eigenfunctions
$q_j(\xi/|\xi|, \xi, \cdot)$ and $\tilde q_j(\xi/|\xi|, \xi, \cdot)$
of $L_\xi$ associated with $\lambda_j$ of form
$q_j=\sum_k \beta_{j,k} v_k$ and $\tilde q_j=\sum_k \tilde \beta_{j,k} \tilde v_k$
where $\{v_j\}$ and $\{\tilde v_j\}$ are dual bases of the total
eigenspace of $L_\xi$ associated with sufficiently small eigenvalues,
analytic in $\omega=\xi/|\xi|$ and $|\xi|$,
with $\tilde v_j(\omega;0)$ constant for $j\ne n$ and $v_n(\omega;0)
\equiv \bar u'(\cdot)$;
$\tilde \beta_{j,1}, \dots, \tilde \beta_{j,n-1},
|\tilde \xi|^{-1}\tilde \beta_{j,n}, \tilde \beta_{j,n+1}$
and $\beta_{j,1}, \dots, \beta_{j,n-1}, |\xi|\beta_{j,n}, \beta_{j,n+1}$
are analytic in $\xi/|\xi|$, $|\xi|$; and $\langle \tilde q_j,q_k\rangle=
\delta_j^k$.
\end{lemma}

\begin{proof}
Recall that $L_\xi$ as an elliptic second-order operator on bounded
domain has spectrum consisting of isolated eigenvalues of finite
multiplicity.
Expanding
\be \label{Lpert}
L_\xi=L_0 + |\xi|L^1_{\xi/|\xi|}+ |\xi|^2L^2_{\xi/|\xi|}
\ee
for each fixed angle $\hat \xi:=\xi/|\xi|$,
consider the continuous family of spectral
perturbation problems in $|\xi|$ indexed by angle $\omega=\xi/|\xi|$
about the eigenvalue $\lambda=0$ of $L_0$.

Because $0$ is an isolated eigenvalue of $L_0$, the associated total
right and left eigenprojections $P_0$ and $\tilde P_0$ perturb
analytically in both $\omega$ and $|\xi|$,
giving projection $P_\xi$ and $\tilde P_\xi$ \cite{K}.
These yield in standard fashion
(for example, by projecting appropriately chosen fixed subspaces)
locally analytic right and left bases $\{v_j\}$ and $\{\tilde v_j\}$
of the associated total eigenspaces given by
the range of $P_\xi$, $\tilde P_\xi$.

Defining $V=(v_1, \dots, v_{n+1})$ and
$\tilde V=(\tilde v_1, \dots, \tilde v_{n+1})^*$, $*$ denoting
adjoint, we may convert the infinite-dimensional
perturbation problem \eqref{Lpert} into an $(n+1)\times (n+1)$
matrix perturbation problem
\be\label{Mpert}
M_\xi=M_0+ |\xi|M_1 + |\xi|^2 M_2,
\ee
where $M_\xi(\omega, |\xi|):= \left<\tilde V_\xi^*, L_\xi V_\xi\right>$ and
$\left<\cdot,\cdot\right>$ refers to the $L^2(x_1)$ inner product on $[0,X]$.
That is, the eigenvalues $\lambda_j(\xi)$
lying near $0$ of $L_\xi$ are the eigenvalues
of $M_\xi$, and the associated right and left eigenfunctions
of $L_\xi$ are
\be\label{vecrel}
f_j=V w_j  \;\hbox{\rm  and } \;
\tilde f_j=\tilde w_j \tilde V^* ,
\ee
where $w_j$ and $\tilde w_j$
are the associated right and left eigenvectors of $M_\xi$.

{\it Case (i)}.
If $\lambda=0$ is a semisimple eigenvalue of $L_0$, then
$M_0=0$, and \eqref{Mpert} reduces to the simpler perturbation
problem $\check M_\xi:=|\xi|^{-1}M_\xi= M_1 + |\xi|M_2$
studied in \cite{OZ4,JZ3}, which $\lambda_j(\xi)=|\xi|\check \lambda_j(\xi)$,
$\check \lambda_j(\xi)$ denoting the eigenvalues of $\check M_\xi$.
Since $\check \lambda_j$ are continous, $\lambda_j$ are differentiable
at $|\xi|=0$ in
the parameter $|\xi|$ as asserted in the introduction.
Moreover, by (H3), the eigenvalues $\check\lambda_j(0)$
of $M_1=\check M_0$ are distinct, and so they perturb analytically
in $\omega$, $|\xi|$, as do the associated right and left eigenvectors.

{\it Case (ii)}.
Hereafter, assume that $\lambda=0$ is a nonsemisimple eigenvalue
of $L_0$, so that $M_0$ is nilpotent but nonzero, possessing a
nontrivial associated Jordan chain.
Moreover, as the $n$-dimensional subspace of
constant functions by direct computation lie
in the kernel of $L_0^*= (\partial_{x_1}^2+A_1^*\partial_{x_1})$,
where $A_1(x_1):=df^1(\bar u(x_1))$, we have that the
$(n+1)$-dimensional zero eigenspace of $L_0$ is consists precisely
of an $n$-dimensional kernel and a single Jordan chain of height two.
Moreover, by translation-invariance (differentiate in $x_1$
the profile equation \eqref{e:profile}), we have $L_0\bar u'=0$,
so that $\bar u'$ lies in the right kernel of $L_0$.

Now, recall the assumption (H2) that
$H: \, \R \times {\cal U} \times \R \times S^{d-1} \times \R^n  \rightarrow \R^n$	
taking
$(X; a, s, \nu, q)  \mapsto u(X; a, s, \nu, q)-a$
is full rank at $(\bar{X}; \bar{u}(0), 0, e_1, \bar{q})$,
where $u(\cdot;\cdot)$ is the solution operator of profile ODE
\eqref{e:profile}.
The fact that $\ker L_0$ is $n$-dimensional implies that
the restriction
$\check H$ taking
$(a, q)  \mapsto u(X; a, s, \nu, q)-a$ for fixed $(X,\nu,s)$
is also full rank, i.e., $H$ is full rank with respect to
the specific parameters $(X,s,\nu)$.
Applying the Implicit Function Theorem and counting dimensions,
we find that the set of periodic solutions, i.e., the inverse
image of zero under map $H$ local to $\bar u$
is a smooth $(n+d+1)$-dimensional manifold
$\{\bar u^a(x\cdot \nu(a)-\alpha-s(a)t)\}$,
with $\alpha\in \RR$, $a\in \RR^{n+d}$.
Moreover, $d+1$ dimensions may be parametrized by $(X,s,\nu)$,
or without loss of generality $(a_1, \dots, a_{d+1}) =(X,s,\nu)$.

Fixing $(X,\nu)$ and $(a_{d+2}, \dots, a_{n+d+1})$,
and varying $s$, we find by differentiation of \eqref{e:profile}
that $f_*:=-\partial_s \bar u$ satisfies\footnote{Note the function $f_*$ is $X$-periodic, and hence in the domain of $L_0$
since we have fixed the period $X$.} the generalized eigenfunction
equation
$$
L_0 f_*= \bar u'.
$$
Thus, $\bar u'$ spans the eigendirection lying at the base of the
Jordan chain, with the generalized zero-eigenfunction of $L_0$
corresponding to variations in speed along the manifold of periodic
solutions about $\bar u$.
Without loss of generality, therefore, we may take
$\tilde v_1, \dots \tilde v_{n-1}$ and $\tilde v_{n+1}$ to
be constant at $|\xi|=0$, i.e., depending only on $\omega=\xi/|\xi|$
and not $x_1$, and $v_{n} \equiv \bar u'$ at $|\xi|=0$
independent of $\omega$.

Recalling from \cite{JZ3} the fact that
$$
\langle c,L^1 \bar u'\rangle=
\langle c,( \omega_1(2\partial_{x_1}-A_1) -\sum_{j\ne 1} \omega_j A_j)
)\bar u'\rangle=
\langle c, \omega_1\partial_{x_1}^2 \bar u  -\sum_{j\ne 1} \omega_j
\partial_{x_1} f^j(\bar u)\rangle
\equiv 0
$$
for any constant functions $c$, where again $\langle \cdot, \cdot\rangle$
denotes $L^2(x_1)$ inner product on the interval $x_1\in [0,X]$,
and $A_j:=df^j(\bar u(\cdot))$,
we find under this normalization that \eqref{Mpert} has the special structure
\ba\label{Mstructure}
M_0=\bp 0_{(n-1)\times (n-1)} & 0_{n-1}& 0_{n-1}\\
0& 0 & 1\\
0& 0 & 0\ep,
\qquad
M_1=\bp * & 0_{n-1}& * \\
*& * & *\\
*& 0 & *\ep.
\ea
Now, rescaling \eqref{Mpert} as
\be\label{rescale}
\check M_\xi:= |\xi|^{-1} S(\xi)M_\xi S(\xi)^{-1},
\ee
where
\be\label{S}
S:=\bp I_{n-1}& 0 & 0\\
0 & |\xi| & 0\\
0 & 0 & 1\\
\ep,
\ee
we obtain
\be\label{checkMpert}
\check M_\xi=
\check M_0 + |\xi|\check M_1 + O(|\xi|^2),
\ee
where $\check M_j= \check M_j(\omega)$ like the original
$M_j$ are analytic matrix-valued functions of $\omega$,
and the eigenvalues $m_j(\xi)=m_j(\omega;|\xi|)$ of $\hat M_\xi$ are
$|\xi|^{-1}\lambda_j(\xi)$.

As the eigenvalues $m_j$ of $\check M_\xi$ are continuous,
the eigenvalues $\lambda_j(\xi)=|\xi|m_j$ are differentiable
at $|\xi|=0$ as asserted in the introduction.
Moreover, by (H3), the eigenvalues $\check\lambda_j(0)$
of $\check M_0$ are distinct, and so they perturb analytically
in $\omega$, $|\xi|$, as do the associated right and left eigenvectors
$z_j$ and $\tilde z_j$.  Undoing the rescaling \eqref{rescale},
and recalling \eqref{vecrel}, we obtain the result.
\end{proof}

\br\label{r:whitrelation}
\textup{
Note that the $n$th coordinate of vectors $w\in \CC^{n+1}$
in the perturbation problem \eqref{Mpert} corresponds as
the coefficient of $\bar u'$ to variations
$\Psi$ in displacement.
Thus, rescaling \eqref{rescale}
amounts to substituting for $\Psi$ the variable
$|\xi|\Psi\sim \Psi_x$ of the Whitham averaged system \eqref{e:wkb}.
}
\er

\section{Linearized stability estimates}\label{linests}
By standard spectral perturbation theory \cite{K}, the total
eigenprojection $P(\xi)$ onto the eigenspace of $L_\xi$
associated with the eigenvalues $\lambda_j(\xi)$, $j=1,\dots, n+1$
described in the previous section
is well-defined and analytic in $\xi$ for $\xi$ sufficiently small,
since these (by discreteness of the spectra of $L_\xi$) are
separated at $\xi=0$ from the rest of the spectrum of $L_0$.
Introducing a smooth cutoff function $\phi(\xi)$ that
is identically one for $|\xi|\le \eps$ and identically
zero for $|\xi|\ge 2\eps$, $\eps>0$ sufficiently small,
we split the solution operator $S(t):=e^{Lt}$ into
low- and high-frequency parts
\be\label{SI}
S^I(t)u_0:=
\Big(\frac{1}{2\pi }\Big)^d \int_{-\pi}^{\pi}\int_{\R^{d-1}}
e^{i\xi \cdot x}
\phi(\xi)P(\xi) e^{L_\xi t}\hat u_0(\xi, x_1)
d\xi_1\, d\tilde \xi
\ee
and
\be\label{SII}
S^{II}(t)u_0:=
\Big(\frac{1}{2\pi }\Big)^d \int_{-\pi}^{\pi}\int_{\R^{d-1}}
e^{i\xi \cdot x}
\big(I-\phi P(\xi)\big)
e^{L_\xi t}\hat u_0(\xi, x_1)
d\xi_1\, d\tilde \xi.
\ee

\subsection{High-frequency bounds}\label{HF}
By standard sectorial bounds \cite{He,Pa} and spectral separation
of $\lambda_j(\xi)$ from the remaining spectra of $L_\xi$,
we have trivially the exponential decay bounds
\ba\label{semigp}
\|e^{L_\xi t}(I-\phi P(\xi))f\|_{L^2([0,X])}
&\le   Ce^{-\theta t}\|f\|_{L^2([0,X])},\\
\|e^{L_\xi t}(I-\phi P(\xi))\partial_{x_1}^l f\|_{L^2([0,X])}
&\le   Ct^{-\frac{l}{2}}e^{-\theta t}\|f\|_{L^2([0,X])},\\
\|\partial_{x_1}^l e^{L_\xi t}(I-\phi P(\xi)) f\|_{L^2([0,X])}
&\le   Ct^{-\frac{l}{2}}e^{-\theta t}\|f\|_{L^2([0,X])},
\ea
for $\theta$, $C>0$, and $0\le m\le K$ ($K$ as in (H1)).
Together with (\ref{iso}), these give immediately the
following estimates.

\begin{proposition}[\cite{OZ4}]\label{p:hf}
Under assumptions (H1)--(H3) and (D1)--(D2),
for some $\theta$, $C>0$,
and all $t>0$, $2\le p\le \infty$, $0\le l\le K+1$, $0\le m\le K$,
\ba\label{SIIest}
\|\partial_x^l S^{II}(t) f\|_{L^2(x)},\;
\|S^{II}(t)\partial_x^l f\|_{L^2(x)}&\le
Ct^{-\frac{l}{2}}e^{-\theta t}\|f\|_{L^2(x)},\\
\|\partial_x^m S^{II}(t)  f\|_{L^p(x)},\;
\|S^{II}(t) \partial_x^m f\|_{L^p(x)}&\le
Ct^{-\frac{d}{2}(\frac{1}{2}-\frac{1}{p})- \frac{m}{2}}
e^{-\theta t}\|f\|_{L^2(x)},
\ea
where, again, derivatives in $x\in\RM^d$ refers to total derivatives.
\end{proposition}

\begin{proof}
The first inequalities follow immediately by (\ref{iso}).
The second follows for $p=\infty$, $m=0$ by Sobolev embedding from
$$
\|S^{II}(t) f\|_{L^\infty(\tilde x; L^2(x_1))}\le
Ct^{-\frac{d-1}{4}}e^{-\theta t}\|f\|_{L^2([0,X])}
$$
and
$$
\|\partial_{x_1} S^{II}(t) f\|_{L^\infty(\tilde x; L^2(x_1))}\le
Ct^{-\frac{d-1}{4} - \frac{1}{2}
}e^{-\theta t}\|f\|_{L^2([0,X])},
$$
which follow by an application of (\ref{iso}) in the $x_1$ variable and
the Hausdorff--Young inequality $\|f\|_{L^\infty(\tilde x)}\le
\|\hat f\|_{L^1(\tilde \xi)}$ in the variable $\tilde x$.
The result for derivatives in $x_1$
and general $2\le p\le \infty$ then follows by
$L^p$ interpolation.  Finally,
the result for derivatives in $\tilde{x}$ follows from the inverse Fourier transform,
equation \eqref{SII}, and the large $|\xi|$ bound
\[
|e^{Lt}f|_{L^2(x_1)}\leq e^{-\theta|\tilde{\xi}|^2 t}|f|_{L^2(x_1)},~|\xi|\textrm{ sufficiently large},
\]
which easily follows from Parseval and the fact that $L_{\xi}$ is a relatively compact perturbation
of $\partial_x^2-|\xi|^2$.  Thus, by the above estimate we have
\begin{align*}
\|e^{Lt}\partial_{\tilde{x}}f\|_{L^2(x)}&\leq C\|e^{L_{\xi}t}|\tilde{\xi}|\hat{f}\|_{L^2(x_1,\xi)}\\
&\leq C\sup\left(e^{-\theta|\tilde{\xi}|^2 t}|\xi|\right)\|\hat{f}\|_{L^2(x_1,\xi)}\\
&\leq Ct^{-1/2}\|f\|_{L^2(x)}.
\end{align*}
A similar argument applies for $1\le m\le K$.
%
%
\end{proof}

\subsection{Low-frequency bounds}\label{LF}

Denote by
\be\label{GI}
G^I(x,t;y):=S^I(t)\delta_y(x)
\ee
the Green kernel associated with $S^I$, and
\be\label{GIxi}
[G^I_\xi(x_1,t;y_1)]:=\phi(\xi)P(\xi) e^{L_\xi t}[\delta_{y_1}(x_1)]
\ee
the corresponding kernel appearing within the Bloch--Fourier representation
of $G^I$, where the brackets on $[G_\xi]$ and $[\delta_y]$
denote the periodic extensions of these functions onto the whole line.
Then, we have the following descriptions of $G^I$, $[G^I_\xi]$,
deriving from the
spectral expansion \eqref{e:surfaces} of $L_\xi$ near $\xi=0$.

\begin{proposition}[\cite{OZ4}]\label{kernels}
Under assumptions (H1)--(H3) and (D1)--(\DDD),
\ba\label{Gxi}
[G^I_\xi(x_1,t;y_1)]&= \phi(\xi)\sum_{j=1}^{n+1}e^{\lambda_j(\xi)t}
q_j(\xi,x_1)\tilde q_j(\xi, y_1)^*,\\
G^I(x,t;y)&=
\Big(\frac{1}{2\pi }\Big)^d \int_{\R^{d}} e^{i\xi \cdot (x-y)}
[G^I_\xi(x_1,t;y_1)] d\xi \\
&=
\Big(\frac{1}{2\pi }\Big)^d \int_{\R^{d}}
e^{i\xi \cdot (x-y)}
\phi(\xi)
\sum_{j=1}^{n+1}e^{\lambda_j(\xi)t} q_j(\xi,x_1)\tilde q_j(\xi, y_1)^*
d\xi,
\ea
where $*$ denotes matrix adjoint, or complex conjugate transpose,
$q_j(\xi,\cdot)$ and $\tilde q_j(\xi,\cdot)$
are right and left eigenfunctions of $L_\xi$ associated with eigenvalues
$\lambda_j(\xi)$ defined in \eqref{e:surfaces},
normalized so that $\langle \tilde q_j,q_j\rangle\equiv 1$.
\end{proposition}

\begin{proof}
Relation  (\ref{Gxi})(i) is immediate from the spectral decomposition
of elliptic operators on finite domains, and the fact that $\lambda_j$
are distinct for $|\xi|>0$ sufficiently small, by (H3).
Substituting (\ref{GI}) into (\ref{SI})
and computing
\be\label{comp1}
\widehat{\delta_y}(\xi,x_1)=
\sum_k e^{2\pi i kx_1}\widehat{\delta_y}(\xi + 2\pi k e_1)=
\sum_k e^{2\pi i kx_1}e^{-i\xi \cdot y-2\pi i ky_1}
= e^{-i\xi \cdot y}[\delta_{y_1}(x_1)],
\ee
where the second and third equalities follow from the fact that
the Fourier transform either continuous or discrete of
the delta-function is unity, we obtain
\ba\label{GIsub}
G^I(x,t;y)&=
\Big(\frac{1}{2\pi }\Big)^d \int_{-\pi}^{\pi}\int_{\R^{d-1}}
e^{i\xi \cdot x} \phi P(\xi) e^{L_\xi t} \widehat{\delta_y}(\xi,x_1)d\xi\\
\nonumber
&=
\Big(\frac{1}{2\pi }\Big)^d \int_{-\pi}^{\pi}\int_{\R^{d-1}}
e^{i\xi \cdot (x-y)}  \phi P(\xi)e^{L_\xi t} [\delta_{y_1}(x_1)] d\xi,
\ea
yielding (\ref{Gxi})(ii) by (\ref{GIxi})(i) and the fact that $\phi$
is supported on $[-\pi,\pi]$.
\end{proof}


\begin{proposition} \label{Gbds}
Under assumptions (H1)-(H3) and (D1)-(\DDD),
the low-frequency Green function $G^I(x,t;y)$ of \eqref{GI} decomposes as
$G^I=E+\tilde G^I$,
\be\label{E1}
E=\bar u'(x)e(x,t;y),
\ee
where, for some $C>0$, all $t>0$,
\ba\label{GIest}
\sup_{y}\|\tilde G^I(\cdot, t,;y) \|_{L^p(x)}
 &\le  C (1+t)^{-\frac{d}{2}(1-\frac{1}{p})}\\
\sup_{y}\|\partial_{y}^r \tilde G^I(\cdot, t,;y) \|_{L^p(x)},
\quad
\sup_{y}\|\partial_{t}^r \tilde G^I(\cdot, t,;y) \|_{L^p(x)}
 &\le  C (1+t)^{-\frac{d}{2}(1-\frac{1}{p})-\frac{1}{2}}\\
\ea
for $p\ge 2$, $1\le r\le 2$,
\ba\label{ederest}
\sup_{y}\| \partial_x^j \partial_t^l
\partial_{y}^r e(\cdot, t,;y) \|_{L^p(x)}
 &\le  C (1+t)^{-\frac{d}{2}(1-\frac{1}{p})- \frac{(j+l)}{2}-\frac{1}{2}}\\
\ea
for $p\ge 2$, $0\le j,k, l$, $j+l\le K$, $1\le r\le 2$,
and
\ba\label{eest}
\sup_{y}\|\tilde \partial_x^j \partial_t^l e(\cdot, t,;y) \|_{L^p(x)}
 &\le  C (1+t)^{-\frac{d}{2}(1-\frac{1}{p})-\frac{(j+l)}{2}}\\
\ea
for $0\le j,k, l$, $j+l\le K$,
provided that $p\ge 2$ and $j+l\ge 1$ or $d\ge 3$,
or $p=\infty$ and $d\ge 1$.
Moreover, $e(x,t;y)\equiv 0$ for $t\le 1$.
\end{proposition}

\br
In Proposition \ref{Gbds}, and throughout the remainder of the paper, derivatives in $y\in\RM^d$
refer to total derivatives, just as with the variable $x\in\RM^d$.
\er

\begin{proof}
In the degenerate case (D3) that $0$ is a semisimple eigenvalue of $L_0$,
these estimates have been established in \cite{OZ4,JZ3}.
Without loss of generality, therefore, we
hereafter assume that $0$ is a nonsemisimple eigenvalue of $L_0$,
with the consequences described in
Lemma \ref{blochfacts}.
Recalling that
\ba\label{Gnew}
G^I(x,t;y)&=
\Big(\frac{1}{2\pi }\Big)^d \int_{\R^{d}}
e^{i\xi \cdot (x-y)}
\phi(\xi)
\sum_{j=1}^{n+1}e^{\lambda_j(\xi)t} q_j(\xi,x_1)\tilde q_j(\xi, y_1)^*
d\xi\\
&=
\Big(\frac{1}{2\pi }\Big)^d \int_{\R^{d}}
e^{i\xi \cdot (x-y)}
\phi(\xi)
\sum_{j,k,l=1}^{n+1}e^{\lambda_j(\xi)t} \beta_{j,k}v_k(\xi,x_1)
\tilde \beta_{j,l}\tilde v_l(\xi, y_1)^*
d\xi,
\ea
define
\ba\label{enew}
\tilde e(x,t;y)&=
\Big(\frac{1}{2\pi }\Big)^d \int_{\R^{d}}
e^{i\xi \cdot (x-y)}
\phi(\xi)
\sum_{j,l}e^{\lambda_j(\xi)t} \beta_{j,n}
\tilde \beta_{j,l}\tilde v_l(\xi, y_1)^* d\xi
\ea
so that
\ba\label{Gdiff}
G^I&(x,t;y) -\bar u'(x_1)\tilde e(x,t;y)=\\
&\Big(\frac{1}{2\pi }\Big)^d \int_{\R^{d}}
e^{i\xi \cdot (x-y)} \phi(\xi)
\sum_{j, k\ne n, l} e^{\lambda_j(\xi)t}
\beta_{j,k} \tilde \beta_{j,l}v_k(\xi,x_1)\tilde v_l(\xi, y_1)^* d\xi\\
&\quad +
\Big(\frac{1}{2\pi }\Big)^d \int_{\R^{d}}
e^{i\xi \cdot (x-y)} \phi(\xi)
\sum_{j, l} e^{\lambda_j(\xi)t}
\beta_{j,n} \tilde \beta_{j,l}
\Big(v_n(\xi,x_1)-\bar u'(x_1)\Big)
\tilde v_l(\xi, y_1)^* d\xi,\\
\ea
where, by analyticity of $v_n$,
$v_n(\xi,x_1)-\bar u'(x_1)=O(|\xi|)$, and so, by Lemma \ref{blochfacts},
\be\label{crucial}
\beta_{j,n} \tilde \beta_{j,l}
\Big(v_n(\xi,x_1)-\bar u'(x_1)\Big)
\tilde v_l(\xi, y_1)^* =O(1)
\ee
and
\be\label{crucial2}
\beta_{j,k} \tilde \beta_{j,l}v_k(\xi,x_1)\tilde v_l(\xi, y_1)^*
=O(1) \; \hbox{\rm for }\; k\ne n.
\ee
Note further that $\tilde v_l\equiv \const$ unless $l=n$,
in which case $\tilde \beta_{jl}=O(|\xi|)$ by Lemma \ref{blochfacts};
hence
\be\label{crucial3}
\partial_{y_1}\Big(\beta_{j,n} \tilde \beta_{j,l}
\Big(v_n(\xi,x_1)-\bar u'(x_1))
\tilde v_l(\xi, y_1)^* \Big) =O(|\xi|)
\ee
and
\be\label{crucial4}
\partial_{y_1}
\Big(\beta_{j,k} \tilde \beta_{j,l}v_k(\xi,x_1)\tilde v_l(\xi, y_1)^* \Big)
=O(|\xi|) \; \hbox{\rm for }\; k\ne n.
\ee

From representation (\ref{Gdiff}), bounds \eqref{crucial}--\eqref{crucial2},
and $\Re \lambda_j(\xi)\le -\theta |\xi|^2$,
we obtain by the triangle inequality
\be
\|G^I-\bar u' \tilde e\|_{L^\infty(x,y)}\le C\|e^{-\theta |\xi|^2 t} \phi(\xi)\|_{L^1(\xi)}
 \le  C (1+t)^{-\frac{d}{2}}.
\ee
Derivative bounds follow similarly,
since $x_1$-derivatives falling on $v_{jk}$ are harmless, whereas,
by \eqref{crucial3}--\eqref{crucial4},
$y_1$- or $t$-derivatives falling on $\tilde v_{jl}$
or on $e^{i\xi\cdot(x-y)}$ bring down a factor
of $|\xi|$ improving the decay rate by factor $(1+t)^{-1/2}$.
(Note that $|\xi|$ is bounded because of the cutoff function $\phi$,
so there is no singularity at $t=0$.)

To obtain bounds for $p=2$, we note that (\ref{GIest} may be viewed
itself as a Bloch--Fourier decomposition with respect to variable
$z:=x-y$, with $y$ appearing as a parameter.
Recalling (\ref{iso}), we may thus estimate
\ba
\sup_y &\|G^I(\cdot,t;y)-\bar u'\tilde e(\cdot, t;y)\|_{L^2(x)}
\le\\
&
C \sum_{j, k\ne n, l}
\sup_y \|\phi(\xi) e^{\lambda_j(\xi)t}
v_k(\cdot, z_1)\tilde v_l^*(\cdot, y_1)
\tilde v_l(\cdot, y_1)^* \|_{L^2(\xi; L^2(z_1\in [0,X]))}\\
&\quad +
C\sum_{j, l} \sup_y \|\phi(\xi) e^{\lambda_j(\xi)t}
\Big( \frac{v_n(\cdot,x_1)-\bar u'(x_1)}{|\cdot|}\Big)
\tilde v_l(\cdot, y_1)^* \|_{L^2(\xi; L^2(z_1\in [0,X]))}\\
&\le
C \sum_{j, k\ne n, l} \sup_y \|\phi(\xi) e^{-\theta |\xi|^2t} \|_{L^2(\xi)}
\sup_\xi\| v_k(\cdot, z_1) \|_{L^2(0,X)}
\| \tilde v_l(\cdot, y_1)^* \|_{L^\infty(0,X)}
\\
&\quad +
C\sum_{j, l} \sup_y \|\phi(\xi) e^{-\theta |\xi|^2t} \|_{L^2(\xi)}
\sup_\xi\| \Big( \frac{v_n(\xi,x_1)-\bar u'(x_1)}{|\xi|}\Big) \|_{L^2(0,X)}
\|\tilde v_l(\cdot, y_1)^* \|_{L^\infty(0,X)} \\
&\le
 C (1+t)^{-\frac{d}{4}},
\ea
where we have used in a crucial way the boundedness of $\tilde v_l$
in $L^\infty$,\footnote{This is clear for $\xi=0$, since $v_j$
are linear combinations of genuine and generalized eigenfunctions,
which are solutions of the homogeneous or inhomogeneous eigenvalue ODE.
More generally, note that resolvent of $L_\xi-\gamma$
gains one derivative, hence the total eigenprojection, as a contour
integral of the resolvent, does too- now, use the one-dimensional
Sobolev inequality for periodic boundary conditions
to bound the $L^\infty$ difference from the
mean by the (bounded) $H^1$ norm, then bound the mean by the $L^1$ norm,
which is controlled by the $L^2$ norm.}
and also the boundedness of
$$
\Big( \frac{v_n(\xi,x_1)-\bar u'(x_1)}{|\xi|}\Big)
\sim
\partial_{|\xi|}v_n(\omega; r)
$$
in $L^2$, where $0<r<|\xi|$.
Derivative bounds follow similarly as above, noting that
$y$- or $t$-derivatives bring down a factor $|\xi|$, while
$x$-derivatives are harmless, to obtain an additional factor
of $(1+t)^{-1/2}$ decay.
Finally, bounds for $2\le p\le \infty$ follow by $L^p$-interpolation.

Defining
\be\label{edef}
e(x,t;y):= \chi(t)\tilde e(x,t;y),
\ee
where $\chi$ is a smooth cutoff function
such that $\chi(t)\equiv 1$ for $t\ge 2$ and $\chi(t)\equiv 0$ for $t\le 1$,
and setting $\tilde G:=G-\bar u'(x_1)e(x,t;y)$,
we readily obtain the estimates \eqref{sheatbds} by combining
the above estimates on $G^I-\bar u \tilde e$
with bound \eqref{SIIest} on $G^{II}$.

Finally, recalling, by Lemma \ref{blochfacts}, that $\tilde v_l\equiv \const$
for $l\ne n$ while $\tilde \beta_{j,n}=O(|\xi|)$, we have
$$
\partial_{y_1} \Big( \beta_{j,n} \tilde \beta_{j,l}\tilde v_l(\xi, y_1)^*\Big)
=o(|\xi|).
$$
Bounds \eqref{ederest} thus
follow from \eqref{enew} by the argument
used to prove \eqref{GIest}, together with the observation that
$x$- or $t$-derivatives bring down factors of $|\xi|$.

Bounds \eqref{eest} follow similarly for $p=\infty$
if $e^{-\theta |\xi|^2t}/|\xi|$ is integrable in $\RR^d$,
and for $p\ge 2$ if $e^{-\theta |\xi|^2t}/|\xi|^2$ is integrable,
thus yielding the stated results for all $d\ge 2$.
In the special case $d=1$, $p=\infty$,
\eqref{Lpert} becomes a simpler one-parameter perturbation
in $\xi$, and the $|\xi|^{-1}$ contributions become
analytic multiples of $\xi^{-1}$, whose principal value
integrals may be carried out explicitly to give a
sum of traveling error functions that is bounded in $L^\infty$;
see the proof of Proposition 1.5, \cite{OZ2} in the one-dimensional case.
We omit this calculation as largely outside our analysis.
(However, note that we need this bound to conclude $L^\infty$
bounded stability in the one-dimensional case.)

\end{proof}

\begin{remark}\label{greenformula}
\textup{
Underlying
our
analysis,
and that of \cite{OZ2,JZ3}, is the fundamental
relation
\be\label{greenform}
G(x,t;y)=
\Big(\frac{1}{2\pi }\Big)^d \int_{-\pi}^{\pi}\int_{\R^{d-1}}
e^{i\xi \cdot (x-y)}[G_\xi(x_1,t;y_1)]d\xi.
\ee
}
\end{remark}

\subsection{Final linearized bounds}\label{s:finallin}

\begin{cor}\label{greenbds}
Under assumptions (H1)--(H3), (D1)--(\DDD),
the Green function $G(x,t;y)$ of \eqref{e:lin} decomposes as
$G=E+\tilde G$,
\be\label{E}
E=\bar u'(x)e(x,t;y),
\ee
where, for some $C>0$, all $t>0$, $1\le q\le 2\le p\le \infty$, $0\le j,k, l$,
$j+l\le K$, $1\le r\le 2$,
\ba\label{sheatbds}
\Big|\int_{-\infty}^{+\infty} \tilde G(x,t;y)f(y)dy\Big|_{L^p(x)}&\le
C (1+t)^{-\frac{d}{2}(1/2-1/p)} t^{-\frac{1}{2}(1/q-1/2)}
|f|_{L^q\cap L^2},\\
\Big|\int_{-\infty}^{+\infty} \partial_y^r \tilde G(x,t;y)f(y)dy\Big|_{L^p(x)}&\le
C (1+t)^{-\frac{d}{2}(1/2-1/p)-\frac{1}{2}+\frac{r}{2}} \\
&\quad \times
t^{-\frac{d}{2}(1/q-1/2)-\frac{r}{2}} |f|_{L^q\cap L^2},\\
\Big|\int_{-\infty}^{+\infty} \partial_t^r \tilde G(x,t;y)f(y)dy\Big|_{L^p(x)}&\le
C (1+t)^{-\frac{d}{2}(1/2-1/p)-\frac{1}{2}+r}
\\ &\quad \times t^{-\frac{d}{2}(1/q-1/2)-r} |f|_{L^q\cap L^2}.\\
\ea
\ba\label{etbds}
\Big|\int_{-\infty}^{+\infty} \partial_x^j\partial_t^k e(x,t;y)f(y)dy\Big|_{L^p}
&\le
(1+t)^{-\frac{d}{2}(1/q-1/p) -\frac{(j+k)}{2} +\frac{1}{2} }|f|_{L^q},\\
\Big|\int_{-\infty}^{+\infty} \partial_x^j\partial_t^k\partial_y^r e(x,t;y)f(y)dy\Big|_{L^p}
&\le
(1+t)^{-\frac{d}{2}(1/q-1/p) -\frac{(j+k)}{2} }|f|_{L^q}.\\
\ea
Moreover, $e(x,t;y)\equiv 0$ for $t\le 1$.
\end{cor}

\begin{proof}
({\it Case $q=1$}).
From (\ref{GIest}) and the triangle inequality we obtain
$$
\Big\|\int_{\R^d}\tilde G^I(x,t;y)f(y)dy\Big\|_{L^p(x)}
\le
\int_{\R^d}\sup_y \|\tilde G^I(\cdot ,t;y)\|_{L^p}|f(y)|dy
\le C(1+t)^{-\frac{d}{2}(1-1/p)}\|f\|_{L^1}
$$
and similarly for $y$- and $t$-derivative estimates, which,
together with \eqref{SIIest}, yield \eqref{sheatbds}.
Bounds \eqref{etbds} follow similarly by the triangle inequality
and \eqref{ederest}--\eqref{eest}.

({\it Case $q=2$}).
From \eqref{crucial}--\eqref{crucial2}, and analyticity of
$v_j$, $\tilde v_j$, we have boundedness
from $L^2[0,X]\to L^2[0,X]$
of the projection-type operators
\be\label{crucialx}
f\to
\beta_{j,n} \tilde \beta_{j,l}
\Big(v_n(\xi,x_1)-\bar u'(x_1)\Big)
\langle \tilde v_l, f\rangle
\ee
and
\be\label{crucial2x}
f\to
\beta_{j,k} \tilde \beta_{j,l}v_k(\xi,x_1)
\langle \tilde v_l,f\rangle
 \; \hbox{\rm for }\; k\ne n,
\ee
uniformly with respect to $\xi$,
from which we obtain by \eqref{Gdiff}, \eqref{edef}, and (\ref{iso})
the bound
\be\label{triv}
\Big|\int_{-\infty}^{+\infty} \tilde G^I(x,t;y)f(y)dy\Big|_{L^2(x)}\le
C\|f\|_{L^2(x)},
\ee
for all $t\ge 0$, yielding together with \eqref{SIIest}
the result \eqref{sheatbds} for $p=2$, $r=1$.
Similarly, by boundedness of $\tilde v_j$, $v_j$, $\bar u'$
in all $L^p[0,X]$, we have
$$
\begin{aligned}
|e^{\lambda_j(\xi)t}
\beta_{j,n} \tilde \beta_{j,l}
\Big(v_n(\xi,x_1)-\bar u'(x_1)\Big)
\langle \tilde v_l, \hat f\rangle |_{L^\infty(x_1)}
& \le Ce^{-\theta |\xi|^2t} |\hat f(\xi,\cdot)|_{L^2(x_1)},\\
|e^{\lambda_j(\xi)t}
\beta_{j,k} \tilde \beta_{j,l}v_k(\xi,x_1)
\langle \tilde v_l, \hat f\rangle |_{L^\infty(x_1)}
& \le Ce^{-\theta |\xi|^2t} |\hat f(\xi,\cdot)|_{L^2(x_1)},
 \; \hbox{\rm for }\; k\ne n,
\end{aligned}
$$
$C,\, \theta>0$, yielding by definitions \eqref{Gdiff}, \eqref{edef} the bound
\ba\label{last}
\Big|\int_{-\infty}^{+\infty} \tilde G^I(x,t;y)f(y)dy\Big|_{L^\infty(x)}
&\le
\Big(\frac{1}{2\pi }\Big)^d \int_{-\pi}^{\pi}\int_{\R^{d-1}}
C\phi(\xi) e^{-\theta |\xi|^2t}|\hat f(\xi,\cdot)|_{L^2(x_1)}
d\xi_1\, d\tilde \xi \\
&\le C|\phi(\xi) e^{-\theta |\xi|^2t}|_{L^2(\xi)} |\hat f|_{L^2(\xi, x_1)}\\
&= C(1+t)^{-\frac{d}{4} } \|f\|_{L^2([0,X])},
\ea
hence giving the result for $p=\infty$, $r=0$.
The result for $r=0$ and general $2\le p\le \infty$ then follows by
$L^p$ interpolation between $p=2$ and $p=\infty$.
Derivative bounds $1\le r\le 2$ follow by
similar arguments, using \eqref{crucial3}--\eqref{crucial4}.
Bounds \eqref{etbds} follow similarly.

({\it Case $1\le q \le 2$}).
By Riesz--Thorin interpolation between the cases $q=1$ and $q=2$,
we obtain the bounds asserted in the general case $1\le q\le 2$,
$2\le p\le \infty$.
\end{proof}

\br\label{lowp}
\textup{
The bounds on $\tilde G$, $e_t$, $e_x$
may be recognized as the
standard diffusive bounds satisfied for the heat equation \cite{Z7}.
For dimension $d=1$, it may be shown using pointwise techniques as
in \cite{OZ2} that the bounds of Corollary
\ref{greenbds} extend to all $1\le q\le p\le \infty$.
}
\er

We note a striking analogy between the Green function decomposition of
Corollary \ref{greenbds} and that of \cite{MaZ3,Z4} in the viscous shock case;
compare Proposition 3.3, \cite{Z7}.

\section{Nonlinear stability in dimension one}\label{s:nonlin}
With the bounds of Corollary \eqref{greenbds}, nonlinear
stability follows by exactly the same argument as in \cite{JZ3},
included here for completeness.
We carry out the nonlinear stability analysis
only in the most difficult, one-dimensional, case.
The extension to the multi-dimensional case is straightforward \cite{JZ3,OZ4}.
(Recall that the nonlinear iteration is easier to close in multi-dimensions,
since the linearized behavior is faster decaying \cite{OZ4,JZ3,S1,S2,S3}.)

Hereafter, take $x\in \RR^1$, dropping the indices
on $f^j$ and $x_j$ and writing $u_t+f(u)_x=u_{xx}$.

\subsection{Nonlinear perturbation equations}\label{s:pert}

Given a solution $\tilde u(x,t)$ of \eqref{eqn:1conslaw},
define the nonlinear perturbation variable
\be\label{pertvar}
v=u-\bar u=
\tilde {u}(x+\psi(x,t))-\bar u(x),
\ee
where
\be\label{uvar}
u(x,t):=\tilde {u}(x+\psi(x,t))
\ee
and $\psi:\RM\times\RM\to\RM$ is to be chosen later.

\begin{lem}\label{4.1}
For $v$, $u$ as in \eqref{pertvar},\eqref{uvar},
\begin{equation}\label{eqn:1nlper}
u_t+f(u)_{x}-u_{xx}=\left(\partial_t-L\right)\bar{u}'(x_1)\psi(x,t)
+\partial_x R + (\partial_t+\partial_x^2)  S ,
\ee
where
\[
R:= v\psi_t + v\psi_{xx}+  (\bar u_x +v_x)\frac{\psi_x^2}{1+\psi_x}
= O(|v|(|\psi_t|+|\psi_{xx}|) +\Big(\frac{|\bar u_x|+|v_x|}{1-|\psi_x|} \Big)|\psi_x|^2)
\]
and
\[
S:=- v\psi_x =O(|v|(|\psi_x|).
\]
\end{lem}

\begin{proof}
To begin, notice from the definition of $u$ in \eqref{uvar} we have by a
straightforward computation
\begin{align*}
u_t(x,t)&=\tilde{u}_x(x+\psi(x,t),t)\psi_t(x,t)+\tilde{u}_t(x+\psi,t)\\
f(u(x,t))_x&=df(\tilde{u}(x+\psi(x,t),t))\tilde{u}_x(x+\psi,t)\cdot(1+\psi_x(x,t))
\end{align*}
and
\begin{align*}
u_{xx}(x,t)&=\left(\tilde{u}_x(x+\psi(x,t),t)\cdot(1+\psi_x(x,t))\right)_x\\
&=\tilde{u}_{xx}(x+\psi(x,t),t)\cdot(1+\psi_x(x,t))+\left(\tilde{u}_x(x+\psi(x,t),t)\cdot\psi_x(x,t)\right)_x.
\end{align*}
Using the fact that
$\tilde u_t + df(\tilde{u})\tilde{u}_x-\tilde{u}_{xx}=0$, it follows that
\ba\label{altform}
u_t+f(u)_{x}-u_{xx}&=\tilde{u}_x\psi_t+df(\tilde{u})\tilde{u}_x\psi_x-\tilde{u}_{xx}\psi_x-\left(\tilde{u}_x\psi_x\right)_x\\
&= \tilde u_x \psi_t
-\tilde u_{t} \psi_x - (\tilde u_x \psi_x)_x
\ea
where it is understood that derivatives of $\tilde u$ appearing
on the righthand side
are evaluated at $(x+\psi(x,t),t)$.
Moreover, by another direct calculation,
using the fact that $L(\bar{u}'(x))=0$ by translation invariance,
we have
\begin{align*}
\left(\partial_t-L\right)\bar{u}'(x)\psi&=\bar{u}_x\psi_t
-\bar{u}_t\psi_{x} -(\bar{u}_x\psi_{x})_{x}.
\end{align*}
Subtracting, and using the facts that,
by differentiation of $(\bar u+ v)(x,t)= \tilde u(x+\psi,t)$,
\ba\label{keyderivs}
\bar u_x + v_x&= \tilde u_x(1+\psi_x),\\
\bar u_t + v_t&= \tilde u_t + \tilde u_x\psi_t,\\
\ea
so that
\ba\label{solvedderivs}
\tilde u_x-\bar u_x -v_x&=
-(\bar u_x+v_x) \frac{\psi_x}{1+\psi_x},\\
\tilde u_t-\bar u_t -v_t&=
-(\bar u_x+v_x) \frac{\psi_t}{1+\psi_x},\\
\ea
we obtain
\begin{align*}
u_t+ f(u)_{x} - u_{xx}&=
(\partial_t-L)\bar{u}'(x)\psi
+v_x\psi_t
- v_t \psi_x - (v_x\psi_x)_x
+ \Big((\bar u_x +v_x)\frac{\psi_x^2}{1+\psi_x} \Big)_x,
\end{align*}
yielding \eqref{eqn:1nlper} by
$v_x\psi_t - v_t \psi_x = (v\psi_t)_x-(v\psi_x)_t$
and
$(v_x\psi_x)_x= (v\psi_x)_{xx} - (v\psi_{xx})_{x} $.
\end{proof}

\begin{cor}
The nonlinear residual $v$ defined in \eqref{pertvar} satisfies
\be\label{veq}
v_t-Lv=\left(\partial_t-L\right)\bar{u}'(x_1)\psi
-Q_{x}+ R_x +(\partial_t+\partial_x^2)S,
\ee
where
\be\label{eqn:Q}
Q:=f(\tilde{u}(x+\psi(x,t),t))-f(\bar{u}(x))-df(\bar{u}(x))v=\mathcal{O}(|v|^2),
\ee
\be\label{eqn:R}
R:= v\psi_t + v\psi_{xx}+  (\bar u_x +v_x)\frac{\psi_x^2}{1+\psi_x},
\ee
and
\be\label{eqn:S}
S:= -v\psi_x =O(|v|(|\psi_x|).
\ee
\end{cor}

\begin{proof}
Taylor expansion comparing \eqref{eqn:1nlper} and
$\bar u_t + f(\bar u)_x-\bar u_{xx}=0$.
\end{proof}

\subsection{Cancellation estimate}\label{s:cancellation}

Our strategy in writing \eqref{veq} is motivated by the following
basic cancellation principle.

\begin{prop}[\cite{HoZ}]\label{p:cancellation}
For any $f(y,s)\in L^p \cap C^2$ with $f(y,0)\equiv 0$, there holds
\be\label{e:cancel}
\int^t_0 \int G(x,t-s;y) (\partial_s - L_y)f(y,s) dy\,ds
= f(x,t).
\ee
\end{prop}

\begin{proof} Integrating the left hand side by parts, we obtain
\be
\int G(x,0;y)f(y,t)dy - \int G(x,t;y)f(y,0)dy
+ \int^t_0 \int
(\partial_t - L_y)^*G(x,t-s;y) f(y,s)dy\, ds.
\label{5.53.2}
\ee
Noting that, by duality,
$$
(\partial_t - L_y)^* G(x,t-s;y) = \delta(x-y) \delta(t-s),
$$
$\delta(\cdot)$ here denoting the Dirac delta-distribution,
we find that the third term on the righthand side
vanishes in \eqref{5.53.2}, while,
because $G(x,0;y) = \delta(x-y)$, the first term is simply $f(x,t)$.
The second term vanishes by $f(y,0)\equiv 0$.
\end{proof}

\subsection{Nonlinear damping estimate}

\begin{proposition}\label{damping}
Let $v_0\in H^K$ ($K$ as in (H1)), and suppose that
for $0\le t\le T$, the $H^K$ norm of $v$
and the $H^K(x,t)$ norms of $\psi_t$ and $\psi_x$
remain bounded by a sufficiently small constant.
There are then constants $\theta_{1,2}>0$ so that, for all $0\leq t\leq T$,
\begin{equation}\label{Ebounds}
|v(t)|_{H^K}^2 \leq C e^{-\theta_1 t} |v(0)|^2_{H^K} +
C \int_0^t e^{-\theta_2(t-s)} \left(|v|_{L^2}^2 +
|(\psi_t, \psi_x)|_{H^K(x,t)}^2 \right) (s)\,ds.
\end{equation}
\end{proposition}

\begin{proof}
Subtracting from the equation \eqref{altform} for $u$
the equation for $\bar u$, we may write the
nonlinear perturbation equation as
\ba\label{vperturteq}
v_t + (df(\bar u)v)_x-v_{xx}= Q(v)_x
+ \tilde u_x \psi_t -\tilde u_{t} \psi_x - (\tilde u_x \psi_x)_x,
\ea
where it is understood that derivatives of $\tilde u$ appearing
on the righthand side
are evaluated at $(x+\psi(x,t),t)$.
Using \eqref{solvedderivs} to replace $\tilde u_x$ and
$\tilde u_t$ respectively by
$\bar u_x + v_x -(\bar u_x+v_x) \frac{\psi_x}{1+\psi_x}$
and
$\bar u_t + v_t -(\bar u_x+v_x) \frac{\psi_t}{1+\psi_x}$,
and moving the resulting $v_t\psi_x$ term to the lefthand side
of \eqref{vperturteq}, we obtain
\ba\label{vperturteq2}
(1+\psi_x) v_t -v_{xx}&=
-(df(\bar u)v)_x+ Q(v)_x
+ \bar u_x \psi_t
\\ &\quad
- ((\bar u_x+v_x)  \psi_x)_x
+ \Big((\bar u_x+v_x) \frac{\psi_x^2}{1+\psi_x}\Big)_x.
\ea

Taking the $L^2$ inner product in $x$ of
$\sum_{j=0}^K \frac{\partial_x^{2j}v}{1+\psi_x}$
against (\ref{vperturteq2}), integrating by parts,
and rearranging the resulting terms,
we arrive at the inequality
\[
\partial_t |v|_{H^K}^2(t) \leq -\theta |\partial_x^{K+1} v|_{L^2}^2 +
C\left( |v|_{H^K}^2
+
|(\psi_t, \psi_x)|_{H^K(x,t)}^2 \right)
,
\]
for some $\theta>0$, $C>0$, so long as $|\tilde u|_{H^K}$ remains bounded,
and $|v|_{H^K}$ and $|(\psi_t,\psi_x)|_{H^K(x,t)}$ remain sufficiently small.
Using the Sobolev interpolation
$
|v|_{H^K}^2 \leq  |\partial_x^{K+1} v|_{L^2}^2 + \tilde{C} | v|_{L^2}^2
$
for $\tilde{C}>0$ sufficiently large, we obtain
$
\partial_t |v|_{H^K}^2(t) \leq -\tilde{\theta} |v|_{H^K}^2 +
C\left( |v|_{L^2}^2+ |(\psi_t, \psi_x)|_{H^K(x,t)}^2 \right)
$
from which (\ref{Ebounds}) follows by Gronwall's inequality.
\end{proof}

\subsection{Integral representation/$\psi$-evolution scheme}

By Proposition \ref{p:cancellation},
we have, applying Duhamel's principle to \eqref{veq},
\ba\label{prelim}
  v(x,t)&=\int^\infty_{-\infty}G(x,t;y)v_0(y)\,dy  \\
  &\quad
  + \int^t_0 \int^\infty_{-\infty} G(x,t-s;y)
  (-Q_y+ R_x + S_t + S_{yy} ) (y,s)\,dy\,ds
+ \psi (t) \bar u'(x).
\ea
Defining $\psi$ implicitly as
\ba
  \psi (x,t)& =-\int^\infty_{-\infty}e(x,t;y) u_0(y)\,dy \\
&\quad
  -\int^t_0\int^{+\infty}_{-\infty} e(x,t-s;y)
  (-Q_y+ R_x + S_t + S_{yy} ) (y,s)\,dy\,ds ,
 \label{psi}
\ea
following \cite{ZH,Z4,MaZ2,MaZ3},
where $e$ is defined as in \eqref{E},
and substituting in \eqref{prelim} the decomposition $G=\bar u'(x)e +  \tilde G$ of Corollary \ref{greenbds},
we obtain the {\it integral representation}
\ba \label{u}
  v(x,t)&=\int^\infty_{-\infty} \tilde G(x,t;y)v_0(y)\,dy \\
&\quad
  +\int^t_0\int^\infty_{-\infty}\tilde G(x,t-s;y)
  (-Q_y+ R_x + S_t + S_{yy} ) (y,s)\,dy\,ds ,
\ea
and, differentiating (\ref{psi}) with respect to $t$,
and recalling that
$e(x,s;y)\equiv 0$ for $s \le 1 $,
\ba \label{psidot}
   \partial_t^j\partial_x^k \psi (x,t)&=-\int^\infty_{-\infty}\partial_t^j\partial_x^k
e(x,t;y) u_0(y)\,dy \\
&\quad
  -\int^t_0\int^{+\infty}_{-\infty} \partial_t^j\partial_x^k
e(x,t-s;y)
  (-Q_y+ R_x + S_t + S_{yy} ) (y,s)\,dy\,ds .
  \ea

Equations \eqref{u}, \eqref{psidot}
together form a complete system in the variables $(v,\partial_t^j \psi,
\partial_x^k\psi)$,
$0\le j\le 1$, $0\le k\le K$,
from the solution of which we may afterward recover the
shift $\psi$ via \eqref{psi}.
From the original differential equation \eqref{veq}
together with \eqref{psidot},
we readily obtain short-time existence and continuity with
respect to $t$ of solutions
$(v,\psi_t, \psi_x)\in H^K$
by a standard contraction-mapping argument based on \eqref{Ebounds},
\eqref{psi}, and and \eqref{etbds}.

%

\subsection{Nonlinear iteration}

Associated with the solution $(u, \psi_t, \psi_x)$ of integral system
\eqref{u}--\eqref{psidot}, define
\ba\label{szeta}
\zeta(t)&:=\sup_{0\le s\le t}
 |(v, \psi_t,\psi_x)|_{H^K}(s)(1+s)^{1/4} .
\ea

\bl\label{sclaim}
For all $t\ge 0$ for which $\zeta(t)$ is finite, some $C>0$,
and $E_0:=|u_0|_{L^1\cap H^K}$,
\be\label{eq:sclaim}
\zeta(t)\le C(E_0+\zeta(t)^2).
\ee
\el

\begin{proof}
By \eqref{eqn:R}--\eqref{eqn:S} and definition \eqref{szeta},
\ba\label{sNbds}
|(Q,R,S)|_{L^1\cap L^\infty}
&\le |(v,v_x,\psi_t,\psi_x)|_{L^2}^2+
|(v,v_x,\psi_t,\psi_x)|_{L^\infty}^2
\le C\zeta(t)^2 (1+t)^{-\frac{1}{2}},\\
\ea
so long as $|\psi_x|\le |\psi_x|_{H^K}\le \zeta(t)$ remains small,
and likewise (using the equation to bound $t$ derivatives in terms
of $x$-derivatives of up to two orders)
\ba\label{sNbds2}
|(\partial_t+\partial_x^2)S|_{L^1\cap L^\infty}
&\le |(v,\psi_x)|_{H^2}^2
+ |(v,\psi_x)|_{W^{2,\infty}}^2
\le C\zeta(t)^2 (1+t)^{-\frac{1}{2}}.\\
\ea

Applying Corollary \ref{greenbds} with $q=1$, $d=1$ to representations
\eqref{u}--\eqref{psidot}, we obtain for any $2\le p<\infty$
\ba\label{sest}
|v(\cdot,t)|_{L^p(x)}& \le
C(1+t)^{-\frac{1}{2}(1-1/p)}E_0 \\
&\quad +
C\zeta(t)^2\int_0^{t} (1+t-s)^{-\frac{1}{2}(1/2-1/p)}(t-s)^{-\frac{3}{4}}
(1+s)^{-\frac{1}{2}}ds\\
&
\le
 C(E_0+\zeta(t)^2) (1+t)^{-\frac{1}{2}(1-1/p)}
\ea
and
\ba\label{sestad}
|(\psi_t,\psi_x)(\cdot, t)|_{W^{K,p}}& \le
C(1+t)^{-\frac{1}{2}}E_0 +
C\zeta(t)^2\int_0^{t} (1+t-s)^{-\frac{1}{2}(1-1/p)-1/2}
(1+s)^{-\frac{1}{2}}ds \\
&
\le
 C(E_0+\zeta(t)^2) (1+t)^{-\frac{1}{2}(1-1/p)}.
\ea
Using \eqref{Ebounds} and \eqref{sest}--\eqref{sestad},
we obtain
$|v(\cdot,t)|_{H^K(x)} \le
 C(E_0+\zeta(t)^2) (1+t)^{-\frac{1}{4}}$.
Combining this with \eqref{sestad}, $p=2$, rearranging, and recalling
definition \eqref{szeta}, we obtain \eqref{sclaim}.
\end{proof}

\begin{proof}[Proof of Theorem \ref{main}]
By short-time $H^K$ existence theory,
$\|(v,\psi_t,\psi_x)\|_{H^{K}}$ is continuous so long as it
remains small, hence $\eta$ remains
continuous so long as it remains small.
By \eqref{sclaim}, therefore,
it follows by continuous induction that
$\eta(t) \le 2C \eta_0$ for $t \ge0$, if $\eta_0 < 1/ 4C$,
yielding by (\ref{szeta}) the result (\ref{eq:smallsest}) for $p=2$.
Applying \eqref{sest}--\eqref{sestad}, we obtain
(\ref{eq:smallsest}) for $2\le p\le p_*$ for any $p_*<\infty$,
with uniform constant $C$.
Taking $p_*>4$ and estimating
$$
|Q|_{L^2}, \, |R|_{L^2}, \, |S|_{L^2}(t)
\le |(v,\psi_t,\psi_x)|_{L^4}^2\le CE_0(1+t)^{-\frac{3}{4}}
$$
in place of the weaker \eqref{sNbds},
then applying Corollary \ref{greenbds} with $q=2$, $d=1$,
we obtain finally \eqref{eq:smallsest} for $2\le p\le \infty$,
by a computation similar \eqref{sest}--\eqref{sestad};
we omit the details of this final bootstrap argument.
Estimate \eqref{eq:stab} then follows using \eqref{etbds} with
$q=d=1$, by
\ba\label{sesta}
|\psi(t)|_{L^p}& \le
C E_0 (1+t)^{\frac{1}{2p}}
+
C\zeta(t)^2\int_0^{t} (1+t-s)^{-\frac{1}{2}(1-1/p)}
(1+s)^{-\frac{1}{2}}ds\\
& \le
C(1+t)^{\frac{1}{2p}}(E_0+\zeta(t)^2),
\ea
together with the fact that
$ \tilde u(x,t)-\bar u(x)= v(x-\psi,t)+ (\bar u(x)-\bar u(x-\psi), $
so that $|\tilde u(\cdot, t)-\bar u|$ is controlled
by the sum of $|v|$ and
$|\bar u(x)-\bar u(x-\psi)|\sim |\psi|$.
This yields stability for $|u-\bar u|_{L^1\cap H^K}|_{t=0}$
sufficiently small, as described in the final line of the theorem.
\end{proof}



\begin{thebibliography}{GMWZ7}

\bibitem[BHZ]{BHZ}
B.~Barker, J.~Humpherys, and K.~Zumbrun.
{\it One-dimensional stability of parallel shock layers
in isentropic magnetohydrodynamics,}
Preprint (2007).

\bibitem[BLZ]{BLZ} B.~Barker, O.~Lafitte, and K.~Zumbrun,
{\it Existence and stability of viscous shock profiles for 2-D isentropic
MHD with infinite electrical resistivity},
preprint (2009).

\bibitem
[G]{G} R. Gardner, {\it On the structure of the spectra of
periodic traveling waves}, J. Math. Pures Appl. 72 (1993), 415-439.

\bibitem
[GZ]{GZ} R. Gardner and K. Zumbrun,
{\it The Gap Lemma and geometric criteria for instability
of viscous shock profiles},
Comm. Pure Appl.  Math. 51 (1998), no. 7, 797--85.

\bibitem[GMWZ1]{GMWZ1}
Gu\`es, O., M\'etivier, G., Williams, M., and Zumbrun, K.,
\emph{Existence and stability of multidimensional shock fronts in
the vanishing viscosity limit}, Arch. Rat. Mech. Anal. 175.
(2004), 151-244.

\bibitem[GMWZ2]{GMWZ2} O. Gues, G. M\'etivier, M. Williams, and K. Zumbrun,
{\it Navier--Stokes regularization of multidimensional Euler shocks.}
Ann. Sci. \'Ecole Norm. Sup. (4)  39  (2006),  no. 1, 75--175.

\bibitem
[He]{He} D. Henry,
{\it Geometric theory of semilinear parabolic equations},
Lecture Notes in Mathematics, Springer--Verlag, Berlin (1981).

\bibitem[HoZ]{HoZ} D. Hoff and K. Zumbrun
{\it Asymptotic behavior of multidimensional scalar viscous shock fronts}, Indiana Univ. Math. Journal, Vol. 49, No. 2 (2000).

\bibitem[HLZ]{HLZ}
J.~Humpherys, O.~Lafitte, and K.~Zumbrun.
{\it Stability of viscous shock profiles in the high Mach number limit,}
Comm. Math. Phys, to appear, 2009.

\bibitem[HLyZ1]{HLyZ1}
J.~Humpherys, G.~Lyng, and K.~Zumbrun.
{\it Spectral stability of ideal-gas shock layers},
Arch. Ration. Mech. Anal., to appear, 2009.

\bibitem[HLyZ2]{HLyZ2}
J.~Humpherys, G.~Lyng, and K.~Zumbrun.
{\it Multidimensional spectral stability of large-amplitude Navier--Stokes
  shocks,}
In preparation.


\bibitem
[JZ1]{JZ1} M. Johnson and K. Zumbrun,
{\it Rigorous Justification of the Whitham Modulation Equations
for the Generalized Korteweg-de Vries Equation,}
preprint (2009).

\bibitem
[JZ3]{JZ3} M. Johnson and K. Zumbrun,
{\it Nonlinear stability and asymptotic behavior of periodic traveling waves
of multidimensional viscous conservation laws in dimensions one and two},
preprint (2009).

\bibitem
[JZB]{JZB} M. Johnson, K. Zumbrun, and J. Bronski,
{\it Bloch wave expansion vs. Whitham Modulation Equations
for the Generalized Korteweg-de Vries Equation,}
in preparation.

\bibitem
[K]{K} T. Kato,
{\it Perturbation theory for linear operators},
Springer--Verlag, Berlin Heidelberg (1985).


\bibitem[MaZ2]{MaZ2} C. Mascia and K. Zumbrun,
{\it Stability of small-amplitude shock profiles of symmetric
hyperbolic-parabolic systems,}
Comm. Pure Appl. Math.  57  (2004),  no. 7, 841--876.

\bibitem[MaZ3]{MaZ3} C. Mascia and K. Zumbrun,
{\it Pointwise Green function bounds for shock profiles of
systems with real viscosity.}
Arch. Ration. Mech. Anal.  169  (2003),  no. 3, 177--263.

\bibitem[MaZ4]{MaZ4} C. Mascia and K. Zumbrun,
\emph{Stability of large-amplitude viscous shock profiles
of hyperbolic-parabolic systems,}
Arch. Ration. Mech. Anal.  172  (2004),  no. 1, 93--131.


\bibitem
[OZ1]{OZ1}
M. Oh and K. Zumbrun, \textit{Stability of periodic
solutions of viscous conservation laws with viscosity-
1. Analysis of the Evans function},
Arch. Ration. Mech. Anal. 166 (2003), no. 2, 99--166.

\bibitem
[OZ2]{OZ2}
M. Oh and K. Zumbrun, \textit{Stability of periodic
solutions of viscous conservation laws with viscosity-
Pointwise bounds on the Green function},
Arch. Ration. Mech. Anal. 166 (2003), no. 2, 167--196.

\bibitem
[OZ3]{OZ3} M. Oh, and K. Zumbrun,
\textit{Low-frequency stability analysis of periodic
traveling-wave solutions of viscous conservation laws
in several dimensions},
Journal for Analysis and its Applications, 25 (2006), 1--21.

\bibitem
[OZ4]{OZ4} M. Oh, and K. Zumbrun,
\textit{Stability and asymptotic behavior of
traveling-wave solutions of viscous conservation laws
in several dimensions}, to appear, Arch. Ration. Mech. Anal.

\bibitem
[Pa]{Pa} A. Pazy, {\it Semigroups of linear operators and applications
to partial differential equations,} Applied Mathematical Sciences, 44,
Springer-Verlag, New York-Berlin, (1983) viii+279 pp. ISBN: 0-387-90845-5.

\bibitem
[S1]{S1} G. Schneider, {\it Nonlinear diffusive stability
of spatially periodic solutions-- abstract theorem and higher space
dimensions},
Proceedings of the International Conference on Asymptotics
in Nonlinear Diffusive Systems (Sendai, 1997),  159--167,
Tohoku Math. Publ., 8, Tohoku Univ., Sendai, 1998.

\bibitem
[S2]{S2} G. Schneider,
{\it Diffusive stability of spatial periodic solutions of the
Swift-Hohenberg equation,} (English. English summary)
Comm. Math. Phys. 178 (1996), no. 3, 679--702.

\bibitem
[S3]{S3} G. Schneider,
{\it Nonlinear stability of Taylor vortices in infinite cylinders,}
Arch. Rat. Mech. Anal. 144 (1998) no. 2, 121--200.

\bibitem
[Se1]{Se1} D. Serre,
{\it Spectral stability of periodic solutions of viscous conservation laws:
Large wavelength analysis}, Comm. Partial Differential Equations 30 (2005),
no. 1-3, 259--282.


\bibitem[Z4]{Z4} K. Zumbrun, {\it Stability of large-amplitude shock
waves of compressible Navier--Stokes equations,}
with an appendix by Helge Kristian Jenssen and Gregory Lyng,
in Handbook of mathematical fluid dynamics. Vol. III,  311--533,
North-Holland, Amsterdam, (2004).

\bibitem[Z6]{Z6} K. Zumbrun,
{\it Dynamical stability of phase
transitions in the p-system with viscosity-capillarity},
SIAM J. Appl. Math. 60 (2000), 1913-1929.

\bibitem[Z7]{Z7} K. Zumbrun,
{\it Instantaneous shock location and one-dimensional nonlinear stability of
viscous shock waves,}
preprint (2009).

\bibitem[ZH]{ZH} K. Zumbrun and P. Howard,
\textit{Pointwise semigroup methods and stability of viscous shock waves}.
Indiana Mathematics Journal V47 (1998), 741--871;
Errata, Indiana Univ. Math. J.  51  (2002),  no. 4, 1017--1021.

\end{thebibliography}
\end{document}